\documentclass{amsart}

\newtheorem{theorem}[equation]{Theorem}
\newtheorem{lemma}[equation]{Lemma}
\newtheorem{corollary}[equation]{Corollary}
\newtheorem{proposition}[equation]{Proposition}

\numberwithin{equation}{section}

\usepackage{amscd,amsmath,amssymb,verbatim}

\begin{document}

\title[$A$-hypergeometric systems that come from
geometry]{${A}$-hypergeometric systems that \\ come from geometry} 
\author{Alan Adolphson}
\address{Department of Mathematics\\
Oklahoma State University\\
Stillwater, Oklahoma 74078}
\email{adolphs@math.okstate.edu}
\author{Steven Sperber}
\address{School of Mathematics\\
University of Minnesota\\
Minneapolis, Minnesota 55455}
\email{sperber@math.umn.edu}
\date{\today}
\begin{abstract}
In recent work, Beukers characterized ${A}$-hypergeometric systems
having a full set of algebraic solutions.  He accomplished this by (1)
determining which ${A}$-hypergeometric systems have a full set of
solutions modulo $p$ for almost all primes $p$ and (2) showing that
these systems come from geometry.  He then applied a fundamental
theorem of N. Katz, which says that such systems have a full set of
algebraic solutions.  In this paper we establish some connections
between nonresonant $A$-hypergeometric systems and de Rham-type
complexes, which leads to a determination of which $A$-hypergeometric
systems come from geometry.  We do not use the fact that the system is
irreducible or find integral formulas for its solutions. 
\end{abstract}
\maketitle

\section{Introduction}

Let $A=\{a^{(1)},\dots,a^{(N)}\}\subseteq{\mathbb Z}^n$ with $a^{(j)} =
(a^{(j)}_1,\dots,a^{(j)}_n)$.  We shall assume throughout this paper
that these lattice points generate ${\mathbb Z}^n$ as abelian group.
Let $L$ be the corresponding lattice of relations, 
\[ L = \bigg\{(l_1,\dots,l_N)\in{\mathbb Z}^N\mid \sum_{j=1}^N l_ja^{(j)} =
0\bigg\}, \]
and let $\alpha = (\alpha_1,\dots,\alpha_n)\in{\mathbb C}^n$.  The
${A}$-{\it hypergeometric system\/} is the system of partial
differential equations in the variables $\lambda_1,\dots,\lambda_N$
consisting of the operators (we write $\partial_j$ for
$\partial/\partial \lambda_j$)
\[ \Box_l = \prod_{l_j>0}\partial_j^{l_j} -
\prod_{l_j<0}\partial_j^{-l_j} \]
for all $l\in L$ and the operators 
\[ Z_{i,\alpha} = \sum_{j=1}^N a^{(j)}_i \lambda_j\partial_j -\alpha_i \]
for $i=1,\dots,n$.  We denote by ${\mathcal D} = {\mathbb C}\langle
\lambda_1,\dots,\lambda_N,\partial_1,\dots,\partial_N\rangle$ the ring
of differential operators in the $\lambda_j$.  The associated
hypergeometric ${\mathcal D}$-module is
\[ {\mathcal M}_\alpha = {\mathcal D}\bigg/\bigg(\sum_{l\in L}{\mathcal D}\Box_l
+ \sum_{i=1}^n{\mathcal D}Z_{i,\alpha}\bigg). \]

Let ${\mathbb C}[\lambda]={\mathbb C}[\lambda_1,\dots,\lambda_N]$ be
the polynomial ring in $N$ variables and let $X$ be a smooth variety
over ${\mathbb C}[\lambda]$.  Let $G$ be a finite group acting on
$X/{\mathbb C}[\lambda]$.  Then $G$ acts on the relative de Rham
cohomology groups $H^i_{\rm DR}(X/{\mathbb C}[\lambda])$.  For an
irreducible representation $\chi$ of $G$, let $H^i_{\rm DR}(X/{\mathbb
  C}[\lambda])^\chi$ denote the $\chi$-isotypic component of $H^i_{\rm
  DR}(X/{\mathbb C}[\lambda])$, i.e., $H^i_{\rm DR}(X/{\mathbb
  C}[\lambda])^\chi$ is the sum of all $G$-submodules of $H^i_{\rm
  DR}(X/{\mathbb C}[\lambda])$ that are isomorphic to $\chi$.  The 
$H^i_{\rm DR}(X/{\mathbb C}[\lambda])^\chi$ are ${\mathcal D}$-modules
via the Gauss-Manin connection.  We say that a ${\mathcal D}$-module
${\mathcal M}$ {\it comes from geometry\/} if it is isomorphic to 
$H^i_{\rm DR}(X/{\mathbb C}[\lambda])^\chi$ for some $(X/{\mathbb
  C}[\lambda], G, \chi)$.

The main idea of this paper is to show that the ${\mathcal M}_\alpha$
are isomorphic as ${\mathcal D}$-modules to certain cohomology groups
that arise in algebraic geometry.  Let $R' = {\mathbb
  C}[\lambda][x_1^{\pm 1},\dots,x_n^{\pm 1}]$, the coordinate ring of
the $n$-torus ${\mathbb T}^n$ over ${\mathbb C}[\lambda]$.  Put
\begin{equation}
f = \sum_{j=1}^N \lambda_jx^{a^{(j)}}\in R' 
\end{equation}
and let $\Omega^k_{R'/{\mathbb C}[\lambda]}$ denote the module of
relative $k$-forms.  We use the set
\[
\bigg\{\frac{dx_{i_1}}{x_{i_1}}\wedge\cdots\wedge\frac{dx_{i_k}}{x_{i_k}}
\mid 1\leq i_1<\dots<i_k\leq n\bigg\} \]
as basis for $\Omega^k_{R'/{\mathbb C}[\lambda]}$ as $R'$-module.  The map
$\nabla_\alpha:\Omega^k_{R'/{\mathbb
    C}[\lambda]}\to\Omega^{k+1}_{R'/{\mathbb C}[\lambda]}$ given by
\[ \nabla_\alpha(\omega) = d\omega + \sum_{k=1}^n
\alpha_i\frac{dx_i}{x_i}\wedge\omega + df\wedge\omega \]
defines a complex $(\Omega^\bullet_{R'/{\mathbb
    C}[\lambda]},\nabla_\alpha)$.  In terms of the above basis, we
have for $\xi\in R'$
\begin{equation}
\nabla_\alpha\biggl(\xi\,\frac{dx_{i_1}}{x_{i_1}}\wedge\cdots\wedge
\frac{dx_{i_k}}{x_{i_k}}\biggr) = \biggl(\sum_{i=1}^n
D_{i,\alpha}(\xi)\,\frac{dx_i}{x_i}\biggr) \wedge
\frac{dx_{i_1}}{x_{i_1}}\wedge\cdots\wedge \frac{dx_{i_k}}{x_{i_k}},
\end{equation}
where
\begin{equation}
D_{i,\alpha} = x_i\frac{\partial}{\partial x_i} + \alpha_i +
x_i\frac{\partial f}{\partial x_i}.
\end{equation}

Formally,
\[ \nabla_\alpha = \frac{1}{x^\alpha\exp f}\circ d\circ x^\alpha\exp
f, \]
so for any derivation $\partial\in{\rm Der}_{\mathbb C}({\mathbb
  C}[\lambda])$ the operator
\[ D_\partial = \frac{1}{x^\alpha\exp f}\circ \partial\circ
x^\alpha\exp f = \partial + f^\partial  \]
on the $\Omega^k_{R'/{\mathbb C}[\lambda]}$ commutes with
$\nabla_\alpha$ (where $f^\partial$ denotes the polynomial obtained
from $f$ by applying $\partial$ to its coefficients).  This defines an
action of ${\rm Der}_{\mathbb C}({\mathbb 
  C}[\lambda])$ on $(\Omega^\bullet_{R'/{\mathbb
    C}[\lambda]},\nabla_\alpha)$.  In particular, $\partial_j$ acts as
$D_j = \partial_j + x^{a^{(j)}}$.  This action extends to an action of
${\mathcal D}$ on $(\Omega^\bullet_{R'/{\mathbb
    C}[\lambda]},\nabla_\alpha)$, which makes the
$H^i(\Omega^\bullet_{R'/{\mathbb C}[\lambda]},\nabla_\alpha)$ into
${\mathcal D}$-modules.

Let $C(A)\subseteq{\mathbb R}^n$ be the real cone generated by $A$ and
let $\ell_1,\dots,\ell_s\in{\mathbb Z}[u_1,\dots,u_n]$ be homogeneous
linear forms defining the codimension-one faces of $C(A)$, normalized
so that the coefficients of each $\ell_i$ are relatively prime and so
that $\ell_i\geq 0$ on $C(A)$ for each $i$.  We say that $\alpha$ is
{\it nonresonant for\/} $A$ if $\ell_i(\alpha)\not\in{\mathbb Z}$ for
all $i$.  Note that $\alpha$ is nonresonant for $A$ if and only if
$\alpha+u$ is nonresonant for $A$ for all $u\in{\mathbb Z}^n$.  (If
$\alpha\in{\mathbb R}^n$, this is equivalent to saying that no proper
face of $C(A)$ contains a point of $\alpha+{\mathbb Z}^n$.)

\begin{theorem}
If $\alpha$ is nonresonant for $A$, then ${\mathcal M}_\alpha\cong
H^n(\Omega^\bullet_{R'/{\mathbb C}[\lambda]},\nabla_\alpha)$ as
${\mathcal D}$-modules. 
\end{theorem}

{\bf Remark.}  It is straightforward to check that for $u\in{\mathbb
  Z}^n$ multiplication by $x^u$ defines an isomorphism of complexes of
${\mathcal D}$-modules
\[ x^u:(\Omega^\bullet_{R'/{\mathbb
    C}[\lambda]},\nabla_{\alpha+u})\to (\Omega^\bullet_{R'/{\mathbb
    C}[\lambda]},\nabla_\alpha) \]
(its inverse is multiplication by $x^{-u}$).  Thus when $\alpha$ is
nonresonant the ${\mathcal M}_{\alpha+u}$ for $u\in{\mathbb Z}^n$ are
all isomorphic as ${\mathcal D}$-modules.

Consider the special case $f = x_ng$, where $g\in{\mathbb
  C}[\lambda][x_1^{\pm 1},\dots,x_{n-1}^{\pm 1}]$.  Let ${\mathbb
  T}^{n-1}$ be the $(n-1)$-torus over ${\mathbb C}[\lambda]$ with
coordinates $x_1,\dots,x_{n-1}$.  Let $U\subseteq{\mathbb T}^{n-1}$ be
the open set where $g$ is nonvanishing and let $\Omega^k_{U/{\mathbb
    C}[\lambda]}$ be the module of relative $k$-forms over the ring of
regular functions on $U$.  The map
$\widetilde{\nabla}_\alpha:\Omega^k_{U/{\mathbb C}[\lambda]}\to
\Omega^{k+1}_{U/{\mathbb C}[\lambda]}$ defined by
\[ \widetilde{\nabla}_\alpha(\omega) = d\omega +
\sum_{i=1}^{n-1}\alpha_i\frac{dx_i}{x_i}\wedge\omega -
\alpha_n\frac{dg}{g}\wedge\omega \]
defines a complex $(\Omega^\bullet_{U/{\mathbb
    C}[\lambda]},\widetilde{\nabla}_\alpha)$.  Formally we have
\[ \widetilde{\nabla}_\alpha =
\frac{g^{\alpha_n}}{x_1^{\alpha_1}\cdots x_n^{\alpha_n}} \circ d \circ
\frac{x_1^{\alpha_1}\cdots x_n^{\alpha_n}}{g^{\alpha_n}}, \]
so if we define for $\partial\in {\rm Der}_{{\mathbb C}}({\mathbb
  C}[\lambda])$
\[ \widetilde{D}_\partial = \frac{g^{\alpha_n}}{x_1^{\alpha_1}\cdots
  x_n^{\alpha_n}} \circ \partial \circ 
\frac{x_1^{\alpha_1}\cdots x_n^{\alpha_n}}{g^{\alpha_n}}  = \partial
- \alpha_n\frac{g^\partial}{g},\]
we get an action of ${\rm Der}_{{\mathbb C}}({\mathbb C}[\lambda])$ on
this complex.  The action extends to an action of~${\mathcal D}$,
making $(\Omega^\bullet_{U/{\mathbb C}[\lambda]},\widetilde{\nabla}_\alpha)$
into a complex of ${\mathcal D}$-modules.  Note that for $u\in{\mathbb
  Z}^n$, multiplication by $x_1^{u_1}\cdots x_{n-1}^{u_{n-1}}/g^{u_n}$
defines an isomorphism of complexes of ${\mathcal D}$-modules
\[ \frac{x_1^{u_1}\cdots
  x_{n-1}^{u_{n-1}}}{g^{u_n}}:(\Omega^\bullet_{U/{\mathbb
    C}[\lambda]},\widetilde{\nabla}_{\alpha+u})\to
(\Omega^\bullet_{U/{\mathbb
    C}[\lambda]},\widetilde{\nabla}_\alpha). \] 

\begin{theorem}
Suppose $f=x_ng(x_1,\dots,x_{n-1})$ and $\alpha$ is nonresonant for $A$.
For all~$i$ there are ${\mathcal D}$-module isomorphisms 
\[ H^i(\Omega^\bullet_{R'/{\mathbb C}[\lambda]},\nabla_\alpha) \cong
H^{i-1}(\Omega^\bullet_{U/{\mathbb
    C}[\lambda]},\widetilde{\nabla}_\alpha). \]
\end{theorem}

When $\alpha\in{\mathbb Q}^n$, it is well known that the
$H^i(\Omega^\bullet_{U/{\mathbb
    C}[\lambda]},\widetilde{\nabla}_\alpha)$ come from geometry.  (We
sketch the proof of this fact in Section~4.)  
From Theorems~1.4 and~1.5, we then get the following result.

\begin{corollary}
Suppose $f=x_ng(x_1,\dots,x_{n-1})$ and $\alpha$ is nonresonant for
$A$.  There is an isomorphism of ${\mathcal D}$-modules ${\mathcal
  M}_\alpha\cong H^{n-1}(\Omega^\bullet_{U/{\mathbb
    C}[\lambda]},\widetilde{\nabla}_\alpha)$.  If in addition
$\alpha\in{\mathbb Q}^n$, then ${\mathcal M}_\alpha$ comes from
geometry. 
\end{corollary}

The proofs of Theorems 1.4 and 1.5 are based on ideas from \cite{A1}
and \cite{AS1}.  Those papers in turn are related to earlier work of
Dwork, Dwork-Loeser, and N. Katz.  We refer the reader to the
introductions of \cite{A1} and \cite{AS1} for more details on the connections
with that earlier work.

\section{Proof of Theorem 1.4}

It is straightforward to check that the $Z_{i,\alpha}$ commute with one another
and that 
\[ \Box_l\circ Z_{i,\alpha} = Z_{i,\beta}\circ\Box_l, \]
where $\beta = \alpha + \sum_{l_j>0} l_ja^{(j)}$ ($ =
\alpha-\sum_{l_j<0} l_ja^{(j)}$ since $l\in L$).  It follows that right
multiplication by $Z_{i,\alpha}$ maps the left ideal $\sum_{l\in
  L}{\mathcal D}\Box_l$ into itself.  If we put ${\mathcal
  P} = {\mathcal D}/\sum_{l\in L}{\mathcal D}\Box_l$, then right
multiplication by the $Z_{i,\alpha}$ is a family of commuting
endomorphisms of ${\mathcal P}$ as left ${\mathcal D}$-module.
Let ${\mathcal C}^\bullet$ be the cohomological Koszul complex
on ${\mathcal P}$ defined by the $Z_{i,\alpha}$.  Concretely,
\[ {\mathcal C}^k = \bigoplus_{1\leq i_1<\dots<i_k\leq n} {\mathcal
  P}\,e_{i_1}\wedge\cdots\wedge e_{i_k}, \]
where the $e_i$ are formal symbols satisfying $e_i\wedge
e_j = -e_j\wedge e_i$ and the boundary operator
$\delta_\alpha:{\mathcal C}^k\to{\mathcal C}^{k+1}$ is defined by additivity
and the formula (for $\sigma\in{\mathcal P}$)
\[ \delta_\alpha(\sigma\,e_{i_1}\wedge\cdots\wedge e_{i_k}) = \sum_{i=1}^n
\sigma Z_{i,\alpha}\,e_i\wedge e_{i_1}\wedge\cdots\wedge e_{i_k}. \]
One obtains a complex of left ${\mathcal D}$-modules $({\mathcal
  C}^\bullet,\delta_\alpha)$ for which 
\begin{equation}
H^n({\mathcal C}^\bullet,\delta_\alpha) = {\mathcal M}_\alpha.
\end{equation}

Let $R = {\mathbb C}[\lambda][x^{a^{(1)}},\dots,x^{a^{(N)}}]$, 
a subring of $R'$ which is also a ${\mathcal D}$-submodule of~$R'$,
and let 
\[ \Omega^k_R\langle\log\rangle = \bigg\{ \sum_{1\leq
  i_1<\dots<i_k\leq n}\xi_{i_1\dots i_k}\,\frac{dx_{i_1}}{x_{i_1}}
\wedge\cdots\wedge\frac{dx_{i_k}}{x_{i_k}} \;\bigg|\; \xi_{i_1\dots
  i_k}\in R\bigg\}\subseteq \Omega^k_{R'/{\mathbb C}[\lambda]}. \]
By (1.2) and (1.3), one has
$\nabla_\alpha(\Omega^k_R\langle\log\rangle) \subseteq
\Omega^{k+1}_R\langle\log\rangle$, so this defines a subcomplex
$(\Omega^\bullet_R\langle\log\rangle,\nabla_\alpha)$ of
$(\Omega^\bullet_{R'/{\mathbb C}[\lambda]},\nabla_\alpha)$.  

The ${\mathbb C}[\lambda]$-module homomorphism $\phi:{\mathcal P}\to
R$ defined by 
\[ \phi(\partial_1^{b_1}\cdots \partial_N^{b_N}) =
x^{\sum_{j=1}^N b_j{a^{(j)}}} \]
is an isomorphism of ${\mathcal
  D}$-modules by \cite[Theorem~4.4]{A1}.  It extends to a map
$\phi:{\mathcal C}^k\to \Omega^k_R\langle\log\rangle$ by
additivity and the formula 
\[ \phi(\sigma\,e_{i_1}\wedge\cdots\wedge e_{i_k}) =
\phi(\sigma)\,\frac{dx_{i_1}}{x_{i_1}}\wedge\cdots\wedge
\frac{dx_{i_k}}{x_{i_k}}. \]
By \cite[Corollary~2.4]{A2}, this is an isomorphism of complexes of
${\mathcal D}$-modules:
\begin{equation}
\phi:({\mathcal C}^\bullet,\delta_\alpha)\xrightarrow{\cong}
(\Omega^\bullet_R\langle\log\rangle,\nabla_\alpha). 
\end{equation}
By (2.1) and (2.2), Theorem 1.4 is a consequence of the following
result.

\begin{proposition}
If $\alpha$ is nonresonant for $A$, then the inclusion map
\[ (\Omega^\bullet_R\langle\log\rangle,\nabla_\alpha)\hookrightarrow
(\Omega^\bullet_{R'/{\mathbb C}[\lambda]},\nabla_\alpha) \]
is a quasi-isomorphism of complexes of ${\mathcal D}$-modules.
\end{proposition}

We state and prove a generalization of Proposition 2.3.
Let $U\subseteq{\mathbb Z}^n$ be a nonempty subset satisfying the
condition: 
\begin{equation}
\text{if $u\in U$, then $u+a^{(j)}\in U$ for $j=1,\dots,N$.}
\end{equation}
If we denote by $M_U$ the free ${\mathbb C}[\lambda]$-module with basis
$\{x^u\mid u\in U\}$, then (2.4) implies that $M_U$ is both an
$R$-submodule and a ${\mathcal D}$-submodule of $R'$.  Note that
$R' = M_{{\mathbb Z}^n}$ and that $R=M_{U_0}$ where $U_0 =
\{\sum_{j=1}^N c_ja^{(j)}\mid c_j\in{\mathbb Z}_{\geq 0}\}$.
Let $\Omega^k_{M_U}\langle\log\rangle\subseteq \Omega^k_{R'/{\mathbb
    C}[\lambda]}$ be the subset
\[ \Omega^k_{M_U}\langle\log\rangle = \bigg\{\sum_{1\leq
  i_1<\dots<i_k\leq n}\xi_{i_1\dots i_k}\,\frac{dx_{i_1}}{x_{i_1}}
\wedge\cdots\wedge\frac{dx_{i_k}}{x_{i_k}} \;\bigg|\; \xi_{i_1\dots
  i_k}\in M_U\bigg\}. \]
By (1.2) and (1.3), we have
$\nabla_\alpha(\Omega^k_{M_U}\langle\log\rangle)\subseteq
\Omega^{k+1}_{M_U}\langle\log\rangle$, so we get a subcomplex
$(\Omega^\bullet_{M_U}\langle\log\rangle,\nabla_\alpha)$ of 
$(\Omega^\bullet_{R'/{\mathbb C}[\lambda]},\nabla_\alpha)$.
Proposition 2.3 is the special case $U=U_0$ of the following more 
general result.  

\begin{proposition}
If $\alpha$ is nonresonant for $A$ and $U$ satisfies $(2.4)$, then the
inclusion 
\[
(\Omega^\bullet_{M_U}\langle\log\rangle,\nabla_\alpha)\hookrightarrow
(\Omega^\bullet_{R'/{\mathbb C}[\lambda]},\nabla_\alpha) \]
is a quasi-isomorphism of complexes of ${\mathcal D}$-modules.
\end{proposition}

As in Section 1, let $\ell_1,\dots,\ell_s\in{\mathbb Z}[u_1,\dots,u_n]$ be
the normalized homogeneous linear forms defining the codimension-one
faces of $C(A)$ and put $S = \{1,\dots,s\}$.  For
$v=(v_1,\dots,v_s)\in{\mathbb Z}^s$ and a subset $T\subseteq S$, put
\[ W(v,T) = \{u\in{\mathbb Z}^n\mid \ell_i(u)\geq v_i \text{ for all
  $i\in T$}\}. \]
Since $\ell_i(a^{(j)})\geq 0$ for all $i,j$, the set $W(v,T)$
satisfies (2.4).  By \cite[Lemma 3.12]{A1}, there exists $v$ such that
$W(v,S)\subseteq U_0$.  It follows that if $U$ satisfies (2.4), then
there exists $v$ such that $W(v,S)\subseteq U$.  

\begin{lemma}
If $\alpha$ is nonresonant for $A$, $U$ satisfies $(2.4)$, and
$W(v,T)\subseteq U$, then the inclusion 
\begin{equation} 
(\Omega^\bullet_{M_{W(v,T)}}\langle
\log\rangle,\nabla_\alpha)\hookrightarrow (\Omega^\bullet_{M_U}\langle
\log\rangle,\nabla_\alpha) 
\end{equation}
is a quasi-isomorphism of complexes of ${\mathcal D}$-modules.
\end{lemma}

Since $W(v,T)\subseteq {\mathbb Z}^n$ for all $v,T$, Lemma 2.6
implies that the inclusion  
\begin{equation}
(\Omega^\bullet_{M_{W(v,T)}}\langle
\log\rangle,\nabla_\alpha)\hookrightarrow(\Omega^\bullet_{R'/{\mathbb 
    C}[\lambda]},\nabla_\alpha) 
\end{equation}
is a quasi-isomorphism for all $v,T$.  The quasi-isomorphisms (2.7)
and (2.8) imply Proposition~2.5.

\begin{proof}[Proof of Lemma $2.6$]
For $T'\subseteq T$, let 
\[ U(v,T')=\{u\in U\mid \text{$\ell_i(u)\geq v_i$ for all $i\in T'$}\}. \]
Note that $U(v,\emptyset) = U$ and that $U(v,T) = W(v,T)$ (since
$W(v,T)\subseteq U$).  By induction, it thus
suffices to show that if $T'\subseteq T$ and $T'' = T'\cup\{t\}$ with
$t\in T\setminus T'$, then the inclusion
\begin{equation}
(\Omega^\bullet_{M_{U(v,T'')}}\langle\log\rangle,\nabla_\alpha)\hookrightarrow
(\Omega^\bullet_{M_{U(v,T')}}\langle\log\rangle,\nabla_\alpha) 
\end{equation}
is a quasi-isomorphism.  Let $Q^\bullet$ be the quotient complex
\[ Q^\bullet = \Omega^\bullet_{M_{U(v,T')}}\langle\log\rangle/
\Omega^\bullet_{M_{U(v,T'')}}\langle\log\rangle. \]
We show that $(2.9)$ is a quasi-isomorphism by showing that
$H^k(Q^\bullet) = 0$ for all $k$.  

We define a filtration $\{F_p\}_{p\leq v_t}$ on the complex 
$\Omega^\bullet_{M_{U(v,T')}}\langle\log\rangle$.  For $p\leq v_t$,
let $F_p \Omega^k_{M_{U(v,T')}}\langle\log\rangle$ be 
the ${\mathbb C}[\lambda]$-submodule spanned by differential forms
\begin{equation}
x^u\,\frac{dx_{i_1}}{x_{i_1}}\wedge\cdots\wedge\frac{dx_{i_k}}{x_{i_k}}
\end{equation}
satisfying $\ell_t(u)\geq p$.  By (1.2) and (1.3), $\nabla_\alpha$
respects this filtration.  We also denote by $F_p$ the induced
filtrations on the subcomplex
$\Omega^\bullet_{M_{U(v,T'')}}\langle\log\rangle$ and the quotient 
complex $Q^\bullet$.  Note that since 
\[ F_{v_t}\Omega^\bullet_{M_{U(v,T')}}\langle\log\rangle =
\Omega^\bullet_{M_{U(v,T'')}}\langle\log\rangle \]
we have $F_{v_t}Q^\bullet = 0$.  To show that $H^k(Q^\bullet) = 0$ for
all $k$, it suffices to show that $H^k({\rm gr}_pQ^\bullet) = 0$ for
all $k$ and $p$, where ${\rm gr}_pQ^\bullet$ denotes the $p$-th graded
piece of the associated graded of the filtration~$F_p$.

Write $\ell_t(u) = \sum_{i=1}^n c_iu_i\in{\mathbb Z}[u]$.
Define $\rho:\Omega^k_{M_{U(v,T')}}\langle\log\rangle\to 
\Omega^{k-1}_{M_{U(v,T')}}\langle\log\rangle$ to be the ${\mathbb
  C}[\lambda]$-module homomorphism satisfying
\[ \rho\biggl(x^u\,\frac{dx_{i_1}}{x_{i_1}}\wedge\cdots\wedge
\frac{dx_{i_k}}{x_{i_k}}\biggr) = x^u\sum_{j=1}^k
(-1)^{j-1}c_{i_j}\,\frac{dx_{i_1}}{x_{i_1}}\wedge\cdots\wedge
\widehat{\frac{dx_{i_j}}{x_{i_j}}}\wedge
  \cdots\wedge\frac{dx_{i_k}}{x_{i_k}}. \]
It is straighforward to check that
\[ (\nabla_\alpha\circ\rho + \rho\circ\nabla_\alpha)\biggl(x^u\,
\frac{dx_{i_1}}{x_{i_1}}\wedge\cdots\wedge
\frac{dx_{i_k}}{x_{i_k}}\biggr) = \biggl(\sum_{i=1}^n
c_iD_{i,\alpha}(x^u)\biggr) \frac{dx_{i_1}}{x_{i_1}}\wedge\cdots\wedge
\frac{dx_{i_k}}{x_{i_k}}, \]
and a calculation using (1.1) and (1.3) shows that
\begin{equation}
\sum_{i=1}^n c_iD_{i,\alpha}(x^u) = \ell_t(\alpha+u) x^u +
\sum_{j=1}^N \lambda_j\ell_t(a^{(j)})x^{u+a^{(j)}}.
\end{equation}
Suppose that the form (2.10) lies in $F_p\setminus F_{p+1}$.
Then $\ell_t(\alpha+u)=\ell_t(\alpha) + p$ and if $\ell_t(a^{(j)})\neq 0$,
then $\ell_t(u+a^{(j)})> p$.  It follows that the second term on the
right-hand side of (2.11) lies in $F_{p+1}$, so on the associated
graded complex the induced map
\[ (\nabla_\alpha\circ\rho + \rho\circ\nabla_\alpha):{\rm gr}_pQ^k\to
   {\rm gr}_pQ^k \]
is just multiplication by $\ell_t(\alpha)+p$.  Since $\alpha$ is
nonresonant for $A$, $\ell_t(\alpha)+p\neq 0$.  Thus multiplication by a
nonzero constant is homotopic to the zero map, which implies that
$H^k({\rm gr}_pQ^\bullet)=0$ for all $k$. 
\end{proof}

\section{Proof of Theorem 1.5}

In this section we assume that $f=x_ng(x_1,\dots,x_{n-1})$.  
By the Remark following Theorem~1.4, we may assume that if
$\alpha_n\in{\mathbb Z}$, then $\alpha_n\geq 1$.  Let $R_+ = {\mathbb
  C}[\lambda][x_1^{\pm 1},\dots,x_{n-1}^{\pm 1}, x_n]$.  By
Proposition~2.5, the inclusion 
\[ (\Omega^\bullet_{R_+}\langle\log\rangle,\nabla_\alpha)
\hookrightarrow (\Omega^\bullet_{R'/{\mathbb
    C}[\lambda]},\nabla_\alpha) \]
is a quasi-isomorphism of complexes of ${\mathcal D}$-modules.
Theorem~1.5 is then a consequence of the following result.

\begin{proposition}
If $\alpha_n\not\in{\mathbb Z}_{\leq 0}$, then there is a
quasi-isomorphism of complexes of ${\mathcal D}$-modules
\[ (\Omega^\bullet_{R_+}\langle\log\rangle,\nabla_\alpha)\to
(\Omega^\bullet_{U/{\mathbb
    C}[\lambda]}[-1],\widetilde{\nabla}_\alpha). \]
\end{proposition}

\begin{proof}
We regard $\Omega^\bullet_{R_+}\langle\log\rangle$ as the total
complex associated to a certain two-row double complex:  let
\[ \Omega^{k,0} = \bigoplus_{1\leq i_1<\dots<i_k\leq n-1}
R_+\,\frac{dx_{i_1}}{x_{i_1}} \wedge\cdots\wedge
\frac{dx_{i_k}}{x_{i_k}} \]
and let
\[ \Omega^{k,1} = \bigoplus_{1\leq i_1<\dots<i_k\leq n-1} R_+\,
\frac{dx_{i_1}}{x_{i_1}} \wedge\cdots\wedge \frac{dx_{i_k}}{x_{i_k}}
\wedge \frac{dx_n}{x_n}.  \]
Let $\partial_h:\Omega^{k,i}\to\Omega^{k+1,i}$ be the map
\[ \partial_h(\omega) = d'\!\omega +
\sum_{i=1}^{n-1}\alpha_i\,\frac{dx_i}{x_i}\wedge\omega + 
x_n(d'\!g\wedge\omega), \]
where $d'$ is exterior differentiation relative to the variables
$x_1,\dots,x_{n-1}$, and let $\partial_v:\Omega^{k,0}\to \Omega^{k,1}$
be the map
\[ \partial_v(\omega) = d''\omega +
\alpha_n\,\frac{dx_n}{x_n}\wedge\omega + g\,dx_n\wedge\omega, \] 
where $d''$ is exterior differentiation relative to the variable
$x_n$.  Since $\partial_v$ is injective, the canonical projection
$\Omega^k_{R_+}\langle\log\rangle\to \Omega^{k-1,1}$ induces a
quasi-isomorphism
\begin{equation}
(\Omega^\bullet_{R_+}\langle\log\rangle,\nabla_\alpha)
\to ((\Omega^{\bullet,1}/\partial_v\Omega^{\bullet,0})[-1],\partial_h)
\end{equation}
(see \cite[Appendix B]{M}).  

To complete the proof of Proposition 3.1, we define a
quasi-isomorphism of complexes between
$(\Omega^{\bullet,1}/\partial_v\Omega^{\bullet,0},\partial_h)$ and
$(\Omega^\bullet_{U/{\mathbb C}[\lambda]},\widetilde{\nabla}_\alpha)$.
Define $\gamma:\Omega^{k,1}\to\Omega^k_{U/{\mathbb C}[\lambda]}$ to
  be the ${\mathbb C}[\lambda]$-module homomorphism satisfying
\[ \gamma\biggl(x^u\,\frac{dx_{i_1}}{x_{i_1}}\cdots
\frac{dx_{i_k}}{x_{i_k}}\,\frac{dx_n}{x_n}\biggr) =
\frac{(-1)^{u_n}(\alpha_n)_{u_n} x_1^{u_1}\cdots
  x_{n-1}^{u_{n-1}}}{g^{u_n}}\,
  \frac{dx_{i_1}}{x_{i_1}}\cdots\frac{dx_{i_k}}{x_{i_k}}, \]
where $(\alpha_n)_{u_n} = \alpha_n(\alpha_n + 1) \cdots
(\alpha_n+u_n-1)$.  It is straightforward to check that $\gamma$
commutes with boundary operators, hence defines a homomorphism of
complexes from $\Omega^{\bullet,1}$ to $\Omega^\bullet_{U/{\mathbb
    C}[\lambda]}$.  The hypothesis that $\alpha_n\not\in{\mathbb
  Z}_{\leq 0}$  implies that $\gamma$ is surjective, and it is
straightforward to check that $\gamma(\partial_v\Omega^{k,0}) = 0$.
It remains only to show that 
\[ \ker(\gamma:\Omega^{k,1}\to\Omega^k_{U/{\mathbb C}[\lambda]}) =
\partial_v\Omega^{k,0}. \]

Let $\xi\in\Omega^{k,1}$ satisfy $\gamma(\xi) = 0$.  Write
\[ \xi = \sum_{i=m}^M x_n^i\xi_i\,\frac{dx_n}{x_n}, \]
where $\xi_i\in\Omega^{k,0}$ is in the ${\mathbb C}[\lambda]$-span of
the forms
\begin{equation}
x_1^{w_1}\cdots x_{n-1}^{w_{n-1}}\,
\frac{dx_{i_1}}{x_{i_1}}\wedge\cdots
\wedge\frac{dx_{i_k}}{x_{i_k}}, \quad 1\leq i_1<\dots<i_k\leq n-1.  
\end{equation}
We prove by induction on $M-m$ that
$\xi\in\partial_v\Omega^{k,0}$.  We have
\[ 0 = \gamma(\xi) =  \sum_{i=m}^M
(-1)^i(\alpha_n)_i\frac{\xi_i}{g^i}. \]
Since $(\alpha_n)_M\neq 0$, this equation implies $\xi_M = g\eta$, where
$\eta$ is in the ${\mathbb C}[\lambda]$-span of the forms (3.3).  It
follows that 
\[ \xi = (-1)^k\partial_v(x_n^{M-1}\eta)
-(\alpha_n+M-1)x_n^{M-1}\eta\,\frac{dx_n}{x_n} +
  \sum_{m=1}^{M-1}x_n^i\xi_i\,\frac{dx_n}{x_n}. \]
By induction we are reduced to the case $\xi =
x_i^m\xi_m\,\frac{dx_n}{x_n}$.  But in this case
\[ 0 = \gamma(\xi) = (-1)^m(\alpha_n)_m\xi_m/g^m, \]
so $\xi_m = 0$.
\end{proof}

If $\alpha_n\not\in{\mathbb Z}$, there is no need to introduce the
complex $\Omega^\bullet_{R_+}\langle\log\rangle$.

\begin{proposition}
If $\alpha_n\not\in{\mathbb Z}$, then there is a quasi-isomorphism of
complexes of ${\mathcal D}$-modules
\[ (\Omega^\bullet_{R'/{\mathbb C}[\lambda]},\nabla_\alpha)\to
(\Omega^\bullet_{U/{\mathbb
    C}[\lambda]}[-1],\widetilde{\nabla}_\alpha). \]
\end{proposition}

\begin{proof}[Sketch of proof]
One proceeds as in the proof of Proposition~3.1 with the complex
$\Omega^\bullet_{R_+}\langle\log\rangle$ replaced by
$\Omega^\bullet_{R'/{\mathbb C}[\lambda]}$.  One defines $\gamma$ as
before with the understanding that for $u_n<0$
\[ (\alpha_n)_{u_n} = ((\alpha_n-1)(\alpha_n-2)\cdots
(\alpha_n+u_n))^{-1}. \]
The proof then proceeds unchanged.  (See \cite[Lemma~2.5]{AS1} for the
details in a similar situation.) 
\end{proof}

\section{Application of N. Katz's results}

We begin by sketching the proof that the $H^i(\Omega^\bullet_{U/{\mathbb
    C}[\lambda]},\widetilde{\nabla}_\alpha)$ come from geometry.  Let
$X\subseteq{\mathbb T}^n_{{\mathbb C}[\lambda]}$ be the hypersurface
$x_n^Dg(x_1^D,\dots,x_{n-1}^D) -1 = 0$, where $D\in{\mathbb Z}_{>0}$.
Since $\Omega^1_{X/{\mathbb C}[\lambda]}$ is a free module with
basis $\{dx_i\}_{i=1}^{n-1}$, ${\mathcal D}$ acts on global $i$-forms
by acting on their coefficients relative to exterior powers of this
basis.  The group $(\boldsymbol{\mu}_D)^n$ acts on $X$ and its relative
de Rham complex $(\Omega^\bullet_{X/{\mathbb C}[\lambda]},d)$.  The
irreducible representations $\chi$ of $(\boldsymbol{\mu}_D)^n$ can be
indexed by $n$-tuples $(a_1,\dots,a_n)$, $0\leq a_i<D$, so that if
$\chi$ corresponds to $(a_1,\dots,a_n)$, then there is an isomorphism
of complexes of ${\mathcal D}$-modules
$(\Omega^\bullet_{X/{\mathbb C}[\lambda]},d)^\chi\cong
(\Omega^\bullet_{U/{\mathbb
    C}[\lambda]},\widetilde{\nabla}_\alpha)$,
where $\alpha = (a_1/D,\dots,a_n/D)$.)  The first assertion of
Corollary~1.6 then implies that 
\begin{equation}
{\mathcal M}_\alpha\cong H^{n-1}_{\rm DR}(X/{\mathbb
  C}[\lambda])^\chi,
\end{equation}
which establishes the second assertion of Corollary~1.6.

Now suppose that $C={\rm Spec}(A)$ is a smooth connected curve over
${\mathbb C}$ and $\phi:C\to {\mathbb A}^N = {\rm Spec}({\mathbb
  C}[\lambda])$ is a morphism.  Let $X'$ be the pullback of $X$ to a
variety over $C$, i.e.,  $X' = A\otimes_{{\mathbb C}[\lambda]}X$.
Then $X'$ is the hypersurface in ${\mathbb T}^n_A$ defined by the equation
$x_n^Dg^\phi(x_1^D,\dots,x_{n-1}^D) -1 = 0$, where $g^\phi\in
A[x_1^{\pm 1},\dots,x_{n-1}^{\pm 1}]$ is the Laurent polynomial obtained
from $g$ by applying $\phi$ to its coefficients (by abuse of notation,
we also denote by $\phi$ the homomorphism ${\mathbb C}[\lambda]\to A$
corresponding to $\phi:C\to{\mathbb A}^N$).  The varieties $X$ and
$X'$ are smooth affine schemes whose de Rham cohomology can be computed as
the cohomology of the complex of global sections of the de Rham
complex.  By the right-exactness of tensor products, one has
\begin{equation} 
\phi^*(H_{\rm DR}^{n-1}(X/{\mathbb C}[\lambda]))\cong
H_{\rm DR}^{n-1}(X'/A). 
\end{equation}
It follows from \cite[Section~14]{K3} that $H_{\rm DR}^{n-1}(X'/A)$
has regular singular points and quasi-unipotent local monodromy at
infinity (i.e., at all points of the quotient field of $A$).
Therefore $H_{\rm DR}^{n-1}(X/{\mathbb C}[\lambda])$ has regular
singular points and quasi-unipotent local monodromy at infinity (in
the sense of \cite[Section~VIII]{K5}).  Equation~(4.1) then implies
that ${\mathcal M}_\alpha$ has regular singular points and
quasi-unipotent local monodromy at infinity. 

To apply the results of \cite{K4}, we observe that the results of this
paper are valid when one replaces ${\mathbb C}[\lambda]$ by ${\mathbb
  C}(\lambda)$.  Let $\bar{\mathcal D}$ denote the ring of differential
operators with coefficients in ${\mathbb C}(\lambda)$ and define
\[ \bar{\mathcal M}_\alpha = \bar{\mathcal D}\bigg/\bigg(\sum_{l\in L}
\bar{\mathcal D}\Box_l + \sum_{i=1}^n \bar{\mathcal
  D}Z_{i,\alpha}\bigg). \]
Put 
\[ \bar{R} = {\mathbb C}(\lambda)[x_1^{\pm 1},\dots,x_n^{\pm 1}], \]
the coordinate ring of the $n$-torus ${\mathbb T}_{{\mathbb
    C}(\lambda)}$.  The proof of Theorem~1.4 establishes the following
result.

\begin{proposition}
If $\alpha$ is nonresonant for $A$, then $\bar{M}_\alpha\cong
H^n(\Omega^\bullet_{\bar{R}/{\mathbb C}(\lambda)},\nabla_\alpha)$ as
$\bar{\mathcal D}$-modules.
\end{proposition}

In the situation of Theorem~1.5, let $\bar{U}\subseteq{\mathbb
  T}_{{\mathbb C}(\lambda)}^n$ be the open set where $g$ is
nonvanishing.  Then we have the following result.

\begin{proposition}
Suppose $f=x_ng(x_1,\dots,x_{n-1})$ and $\alpha$ is nonresonant for $A$.
For all $i$ there are $\bar{\mathcal D}$-module isomorphisms
\[ H^i(\Omega^\bullet_{\bar{R}/{\mathbb
    C}(\lambda)},\nabla_\alpha)\cong
H^{i-1}(\Omega^\bullet_{\bar{U}/{\mathbb
    C}(\lambda)},\tilde{\nabla}_\alpha). \]
\end{proposition}

Combining these propositions gives the following result.

\begin{corollary}
Suppose $f=x_ng(x_1,\dots,x_{n-1})$ and $\alpha$ is nonresonant for
$A$.  There is an isomorphism of $\bar{\mathcal D}$-modules
$\bar{\mathcal M}_\alpha\cong H^{n-1}(\Omega^\bullet_{\bar{U}/{\mathbb
    C}(\lambda)},\tilde{\nabla}_\alpha)$.  If in addition
$\alpha\in{\mathbb Q}^n$, then $\bar{\mathcal M}_\alpha$ comes from
geometry. 
\end{corollary}

Explicitly, letting $\bar{X}\subseteq{\mathbb T}^n_{{\mathbb
    C}(\lambda)}$ be the hypersurface
$x_n^Dg(x_1^D,\dots,x_{n-1}^D)-1=0$, we have (corresponding to
Equation~(4.1))
\begin{equation}
\bar{\mathcal M}_\alpha\cong H^{n-1}_{\rm DR}(\bar{X}/{\mathbb
  C}(\lambda))^\chi.
\end{equation}
By \cite[Theorem 5.7]{K4} $H^{n-1}_{\rm DR}(\bar{X}/{\mathbb
  C}(\lambda))^\chi$ has a full set of polynomial solutions modulo~$p$
for almost all primes $p$ if and only if it has a full set of algebraic
solutions.  Note that the solution sets of ${\mathcal M}_\alpha$ and
$\bar{\mathcal M}_\alpha$ in the algebraic closure of ${\mathbb
  C}(\lambda)$ are identical.  From Equation~(4.6), we then get the
following result.  

\begin{corollary}
Suppose $f=x_ng(x_1,\dots,x_{n-1})$ and $\alpha\in{\mathbb Q}^n$ is
nonresonant for $A$.  The hypergeometric ${\mathcal D}$-module
${\mathcal M}_\alpha$ has a full set of polynomial solutions modulo
$p$ for almost all primes $p$ if and only if it has a full set of
algebraic solutions. 
\end{corollary}


\begin{thebibliography}{99}

\bibitem{A1} Adolphson, Alan. Hypergeometric functions and rings
  generated by monomials. Duke Math.~J. {\bf 73} (1994), no.\ 2,
  269--290.  

\bibitem{A2} Adolphson, Alan. Higher solutions of hypergeometric
  systems and Dwork cohomology. Rend.\ Sem.\ Mat.\ Univ.\ Padova {\bf
    101} (1999), 179--190. 

\bibitem{AS1} Adolphson, Alan; Sperber, Steven. On twisted de Rham
  cohomology. Nagoya Math.\ J. {\bf 146} (1997), 55--81.

\bibitem{B} Beukers, Frits.  Algebraic $A$-hypergeometric
  functions.  Invent.\ Math.\ {\bf 180} (2010), no.\ 3, 589--610

\bibitem{D} Dwork, Bernard. Generalized hypergeometric
  functions. Oxford Mathematical Monographs. Oxford Science
  Publications. The Clarendon Press, Oxford University Press, New
  York, 1990. 

\bibitem{DL} Dwork, B.; Loeser, F. Hypergeometric
  series. Japan.\ J. Math.\ (N.S.) {\bf 19} (1993), no.\ 1, 81--129.  

\bibitem{K1} Katz, Nicholas. Thesis (1966), Princeton University.

\bibitem{K2} Katz, Nicholas. On the differential equations satisfied
  by period matrices. Inst.\ Hautes \'{E}tudes
  Sci.\ Publ.\ Math.\ No.\ 35 (1968), 223--258. 

\bibitem{K3} Katz, Nicholas. Nilpotent connections and the monodromy
  theorem: Applications of a result of Turrittin. Inst.\ Hautes \'{E}tudes
  Sci.\ Publ.\ Math.\ No.\ 39 (1970), 175--232. 

\bibitem{K4} Katz, Nicholas. Algebraic solutions of differential
  equations ($p$-curvature and the Hodge filtration). Invent.\ Math.\
  {\bf 18} (1972), 1--118. 

\bibitem{K5} Katz, Nicholas. A conjecture in the arithmetic theory
  of differential equations. Bull.\ Soc.\ Math.\ France {\bf 110} (1982),
  no.\ 2, 203--239. 

\bibitem{K6} Katz, Nicholas M. Corrections to: ``A conjecture in the
  arithmetic theory of differential
  equations''. Bull.\ Soc.\ Math.\ France {\bf 110} (1982), no.\ 3, 
  347--348. 

\bibitem{M} Matsumura, Hideyuki. Commutative ring theory. Translated
  from the Japanese by M. Reid. Cambridge Studies in Advanced
  Mathematics, 8. Cambridge University Press, Cambridge, 1986.  

\end{thebibliography}
\end{document}